\documentclass[a4paper,20pt]{amsproc}
  \usepackage{epsfig,color}
\usepackage{amsmath}
\usepackage{bbm}
\newtheorem{theorem}{Theorem}[section]

\newtheorem{corollary}[theorem]{Corollary}

\theoremstyle{definition}
\newtheorem{definition}[theorem]{Definition}
\newtheorem{example}[theorem]{Example}

\newtheorem{question}[theorem]{Question}

\theoremstyle{remark}

\numberwithin{equation}{section}

\newcommand{\R}{\mathbb{R}}

\newcommand{\K}{\mathbb{K}}
\newcommand{\lowK}{\mathbbm{k}}

\begin{document}

\title{Extension of local smooth maps of Banach spaces}

\author{Genrich Belitskii}
\address{Department of Mathematics and Computer Science, Ben Gurion University of the Negev, P.O.B. 653, Beer Sheva, 84105, Israel}
\email{genrich@cs.bgu.ac.il}

\author{Victoria Rayskin}
\address{Department of Mathematics, Tufts University, Medford, MA 02155-5597}
\email{victoria.rayskin@tufts.edu}




\begin{abstract}
 It is known that smooth bump functions are absent in the majority of infinite-dimensional Banach spaces. This is an obstacle in the development of local analysis, in particular in the questions of extending local maps onto the whole space. We suggest an approach that substitutes bump functions with special maps, which we call K-maps. It allows us to extend smooth local maps from non-smooth spaces, such as $C^q[0,1], q=0,1,...$. We also prove the Borel lemma for spaces possessing K-maps.

\end{abstract}

\maketitle

\section{Introduction}\label{sec-intro}
With the advancement of dynamical systems and analysis, the complexity of global analysis became evident. This stimulated the development of techniques for the study of local properties of a global problem. One of the methods of localization is based on the functions with bounded support. The history of applications of functions vanishing outside of a bounded set goes back to the works of Sobolev (\cite{S}) on generalized functions.\footnote{Sobolev was a student of N.M. Guenter, and this work was probably influenced by Guenter. However, Guenter was accused in the development of "abstract" science at the time when the USSR was desperate for an applied theory for the creation of atomic weapons. Guenter was forced to resign from his job.} Later,  functions with bounded support were used by Kurt Otto Friedrichs in his paper of 1944. His colleague, Donald Alexander Flanders, suggested the name mollifiers. Friedrichs himself acknowledged Sobolev's work on mollifiers stating that: "These mollifiers were introduced by Sobolev and the author". A special type of mollifier, which is equal to 1 in the area of interest and smoothly vanishes outside of a bigger set, we call a  bump function.
\par
The bump functions are used in the works studying local properties of dynamical systems in $\R^n$ in the book of Nitecki (\cite{N}). In this book, the Lemma on page 79 asserts that a diffeomorphism conjugates with its linear part in some small neighborhood of a hyperbolic fixed point. The proof is based on the assumption that the diffeomirphism can be modified outside of the neighborhood of the fixed point. Indeed, in $\R^n$ this can be achieved with the help of the bump function.
\par
There is, however, an obstacle in the local analysis of dynamical systems in infinite-dimensional spaces. The majority of infinite-dimensional Banach spaces do not have smooth bump functions. In the works \cite{B}, \cite{B-R}, \cite{R} we discuss the conditions when two $C^{\infty}$ diffeomorphisms on Banach spaces, possessing K-maps, are locally $C^{\infty}$-conjugate. K-maps are the maps that substitute bump functions and allow localization of Banach spaces.
\par
The main objectives of this paper are to present K-maps, to introduce the questions of existence of smooth K-maps on various infinite dimensional Banach spaces and their subsets and to prove the Borel Lemma for the spaces that admit smooth K-maps.

Let us recall the notion of a germ at a point. Let $X$ be a real Banach space, $Y$ be either a real or complex one, and $z\in X$ be a  point. Two maps $f_1$ and $f_2$ from neighborhoods $U_1$ and $U_2$ of the point $z$ into $Y$ are called {\it equivalent} if there is a neighborhood $V \subset U_1\cap U_2$   such that both of the maps coincide on $V$. A {\it germ at $z$} is an equivalence class.
Therefore, every local map $f$ from a neighborhood of $z$ into $Y$ defines a germ at $z$. Sometimes in the literature it is denoted by $[f]$, although in general we will use the same notation, $f$, as for the map.
 \par
We consider Fr{\'e}chet $C^q$-maps with $q=0,1,2...,\infty$. All notions and notations of differential calculus in Banach spaces we borrow from \cite{C}.
 \par
 For a given $C^q$-germ $f$ at $z$, we pose the following questions.
 \begin{question}\label{quest-exists-global}
 Does its global $C^q$-representative (i.e., a $C^q$-map defined on the whole $X$) exist?
 \end{question}
\begin{question}\label{quest-exists-global-bounded}
Assume that $f$ has  local representatives with bounded derivatives. Does there exist a global one with the same property?
\end{question}
Below, without loss of generality we assume that $z=0$.
\par
Usually for extending local maps described in Question~\ref{quest-exists-global} and Question~\ref{quest-exists-global-bounded} bump functions are used. The classical definition of a bump function (\cite{S}) is a non-zero bounded $C^q$-function from $X$ to $\R$ having a bounded support. We use a similar modified definition which is more suitable for our aims. Namely, a bump function at $0$ is a $C^q$-map $\delta_U:X \to \R$  which is equal to $1$ in a neighborhood  of $0$ and vanishing outside of a lager neighborhood $U$. If $f$ is a local representative of a $C^q$-germ defined in a neighborhood $V$ ($\overline{U}\subset V$) then

\begin{equation}\label{eqn-F} 
 F(x)=\left\{
\begin{array}{ll}
\delta_U(x)f(x),& x\in U\\
0,& x\notin U
\end{array}
\right.
\end{equation}

is a global $C^q$-representative of the germ $f$, and it solves at least the Question~\ref{quest-exists-global}. If, in addition, all derivatives of $\delta_U$ are bounded on the entire $X$, then (\ref{eqn-F}) solves both of the Questions. If these functions do exist for any $U$, then every germ has a global representative.

A continuous bump function exists in any Banach space. It suffices to set 
$\delta(x)=\tau(||x||)$, where $\tau$ is a continuous bump function at zero on the real line.
    Let $p=2n$ be an even integer. Then
                       $$  \delta(x)=\tau(||x||^p)$$
is a $C^\infty$-bump function at zero on $l_p$. Here $\tau$ is a $C^\infty$-bump function on the real line.
 
However, if $p$ is not an even integer, then $l_p$ space does not have $C^q$-smooth ($q>p$) bump functions (see~\cite{M}). The Banach-Mazur theorem states that any real separable Banach space is isometrically isomorphic to a closed subspace of $C[0,1]$.  Consequently, the space of $C[0,1]$ does not have smooth bump functions at all\footnote{Originally absence of smooth bump functions on $C[0,1]$ was proved in~\cite{K}, 24 years before the publication~\cite{M}.}. Spaces possessing $C^q$-bump functions at zero are called $C^q$-smooth (see~\cite{M}).
 
\begin{example}\label{example-integral}
The real function
           $$
          f(x)=\int_0^1 \frac{dt}{1-x(t)},\ \    x\in C[0,1]
 $$
defines a $C^\infty$ (which is even analytic) germ at zero. In spite of the absence of smooth bump functions, the germ has a
global $C^\infty$ representative. To show this, let $h$ be  a real $C^\infty$-function on the real line such that
\begin{equation}\label{eqn-h}
          h(s)=\left\{
\begin{array}{ll}
s, &|s|<1/3\\
0, &|s|>1/2. 
\end{array}
\right.
\end{equation}

Then the $C^\infty$-function
 $$
                F(x)=\int_0^1 \frac{dt}{1-h(x(t))}
 $$
coincides with $f$ in the ball $||x||<1/3$ and is a global representative of the germ with bounded derivatives of all orders.
\end{example}

\section{K-maps}\label{sec-K-maps}
\begin{definition}\label{def-K-map} A  {\it $C^q$-K-map} on a Banach space $X$ is a  global bounded $C^q$-representative of the germ of the identity map at the
origin.
\end{definition}

     In other words, it is a $C^q$-map $H: X \to X$ which coincides with identity in a neighborhood of zero and such that
 $$
                               \sup_x || H(x)||<\infty.
 $$
The following assertion explains the meaning of the notion.
 
\begin{theorem}\label{thm-extension} Let a space $X$ possesses a $C^q$-K-map $H$. Then for every Banach space $Y$ and any $C^q$-germ $f$ at zero from $X$ to $Y$ there exists a global $C^q$-representative. Moreover, if all derivatives of $H$ are bounded, and $f$ contains a local representative bounded together with all its derivatives, then it has a global one with the same property.
\end{theorem}

\begin{proof} Let $||H(x)||<N,\ x\in X$. Then the image $H(X)$ is contained in the open ball $U=\{x: ||x||<N\}$. Further, let $f$ be a representative of the germ defined on a neighborhood $U$, and let a closed ball $B_\epsilon=\{x:||x||\le\epsilon<N  \}\subset U$. The map
\begin{equation}
              H_1(x)=\frac{\epsilon}{N} H\left(\frac{N}{\epsilon} x\right)
\end{equation}
is a $C^q$-K-map also, and its image is contained in $U$. Therefore the map
 \begin{equation}\label{eqn-global-F}       
F(x)=f(H_1(x))
\end{equation}
is well-defined on the whole space $X$, and is a global $C^q$-representative of the germ $f$. If both of the maps $H$ and $f$ are bounded together with all of their derivatives, then $F$ possesses the same property. This completes the proof.
 \end{proof}
  
\section{Examples}\label{sec-examples}
Let us present spaces having K-maps.
 
\begin{itemize} 
\item[1.] Let $X$ be $C^q$-smooth, and let $\delta (x)$ be a $C^q$-bump function at zero. Then
 $$
                   H(x)=\delta(x) x
 $$
is a $C^q$-K-map. If the bump function is bounded together with all its derivatives, then $H$ has the same property.
 
\item[2.] Let  $X=C(M)$ be a space (a Banach algebra) of all continuous functions on a compact Hausdorff space $M$ with
 $$
                     ||x||=\max_t|x(t)| ,\ \  t\in M,
 $$
and let h be a $C^\infty$-bump function on the real line. Then the map
 \begin{equation}\label{C[01]-K-map}
                                     H(x)(t)=h(x(t))x(t),\ \  x\in X
 \end{equation}
is $C^\infty$-K-map with bounded derivatives of all orders.

\item[3.] More generally, let $X\subset C(M)$ be a subspace such that
 
\begin{equation}\label{ideal-property}
x(t)\in X \implies h(x(t))x(t)\in X
 \end{equation}
for any $C^\infty$  bump function $h$ on the real line.

 Then~(\ref{C[01]-K-map}) defines a $C^\infty$-K-map with bounded derivatives. For example, any ideal $X$ of the algebra satisfies (\ref{ideal-property}).

\item[4.] Let $X=C^n(M)$ be a space (which is also a Banach algebra)  of all $C^n$-functions on a smooth compact manifold $M$ with or without boundary, and
$$
         ||x||=\max_k\max_t ||x^{(k)}(t)||,\  k\leq n,\ t\in M.
 $$
Then (\ref{C[01]-K-map}) gives a $C^\infty$-K-map with bounded derivatives. The same holds for any closed subspace $X\subset C^n(M)$ satisfying (\ref{ideal-property}). As above, $X$ may be an ideal of the algebra.
 \end{itemize}

\begin{corollary}
Let  a space $X$ be as in the examples of items 1-4. Then for any Banach space $Y$ and any $C^q$-germ at zero there is a global $C^q$-representative. If the germ contains a local representative with bounded derivatives, then there is a global one with the same property.
\end{corollary}

\section{The Borel lemma for Banach spaces}\label{sec-Borel-lemma}
Let $X$ be a linear space over a field $\lowK$ ($char \ \lowK =0$) and $Y$ be a linear space over a field $\K$ ($\lowK\subset\K$). A map $P_j:X \to Y$ is called polynomial homogeneous map of degree $j$ if there is a $j$-linear map
$$
                   g: 
\underbrace{
X\times X\times ...\times X 
}_{j}
\to Y
$$
such that $P_j(x)=g(x,x,...,x)$. 
The map $g$ is not unique, but there is a unique symmetric one. We will assume that $g$ is symmetric.
Then, the first derivative of $P_j$ at a point $z\in X$, is a linear map $X \to Y$, and can be calculated by the formula
 $$
          P'_j(z)(x)=jg(z,......,z,x)
 $$
In general, the derivative of order $n\leq j$, is a homogeneous polynomial map of degree $n$ and equals
 $$
          P_j^{(n)}(z)(x)^n=j(j-1)....(j-n+1)g(z,...,z,x,...x)
 $$
 In particular,
 $$
          P_j^{(j)}(z)(x)^j=j!g(x,...,x)=j!P_j(x)
 $$
does not depend on $z$. And lastly,  for $n>j$, $P_j^{(n)}(z)=0$.
 
It follows that all derivatives of $P_j$ at zero are zero, except for the order $j$.
 
The latter equals  $P_j^{(j)}(0)(x)=j!P_j(x)$.

Let now $X$ and $Y$ be Banach spaces. Recall that $X$ must be real, while $Y$ can be real or complex. Let $f: X\to Y$ be  a local $C^\infty$ map. Then
 \begin{equation}\label{P-j}
                        P_j(x)=f^{(j)}(0)(x)^j
 \end{equation}
for any $j=0,1,....$ is a continuous polynomial map of degree $j$. Therefore, for some $c_j>0$ we have the estimate
\begin{equation}\label{estimate}
||P_j(x)||\leq c_j ||x||^j ,\   \ x\in X.
\end{equation}
 
\begin{question}\label{surj} Given a sequence $\{P_j\}_{j=0}^{\infty}$ of continuous polynomial maps from $X$ to $Y$ of degree $j$, does there exist a $C^\infty$-germ $f: X \to Y$, which satisfies (\ref{P-j}) for all $j=0,1,...$? \end{question}

The classical Borel lemma states that, given a sequence of real numbers $\{a_n\}$, there is a $C^\infty$ function $f$ on the real
line such that  $f^{(n)}(0)=a_n$. This means that answer to the Question~\ref{surj} is positive for $X=Y=\R$. The same is true for finite-dimensional $X$ and $Y$.

\begin{theorem}[The Borel lemma] Let a Banach space X possesses a $C^\infty$-K-map with bounded derivatives of all orders.
Then for any Banach space Y and any sequence $\{P_j\}_{j=0}^{\infty}$ of continuous homogeneous polynomial maps from $X$ to $Y$ there is a
$C^\infty$-map from $X$ to $Y$ with bounded derivatives of all orders such that (\ref{P-j}) is satisfied for all $j=0,1,...$
 \end{theorem}

\begin{proof} Let $H$ be a $C^\infty$-K-map at zero with bounded derivatives on $X$. Set $H_j(x)=\epsilon_j H(x/\epsilon_j)$. For a given $\epsilon_j$ the map $H_j(x)$ is a K-map. 
 
Then the map $P_j(H_j(x))$ belongs to $C^\infty(X,Y)$, and all its derivatives at $0$ are zero, except for  the order $j$. The latter equals to $P_j^{(j)}(0)(x)^j=j!P_j(x)$. In addition, all derivatives of the map are bounded, and the derivative of order $n$ allows the following estimate
 $$
                        ||(P_j(H_j(x))^{(n)}||\leq \epsilon_j^{j-n}c_{j,n}
$$
with constants $c_{j,n}$ depending only on the maps $P_j$, $H$, and not depending on a choice of $\epsilon_j$. 
Therefore, under an appropriate choice of $\epsilon_j$ the series
 $$
     f(x)=\sum_0^\infty \frac{1}{j!} P_j\left(H_j(x)\right)
$$
converges in $C^\infty$ topology to a map from $X$ to $Y$. It is clear that
                      $$f^{(n)}(0)(x)^n=P_n(x).$$
This equality proves the statement.
\end{proof}
 
\begin{corollary} Let a space X  be as in the examples listed in items 1-4 of Section~\ref{sec-examples}. 
Then for any Banach space Y and any sequence $\{P_j\}_{j=0}^{\infty}$ of continuous homogeneous polynomial maps from $X$ to $Y$ there is a
$C^\infty$-map from $X$ to $Y$ with bounded derivatives of all orders such that (\ref{P-j}) is satisfied for all $j=0,1,...$
\end{corollary}

\section{Open questions}\label{sec-questions}
\begin{itemize}
\item[1.] For which spaces do $C^\infty$-K-maps exist? Do Banach spaces without $C^{\infty}$-K-maps  exist? 
In particular, do they exist on $l_p$, with non-even $p$?

\item[2.] How to extend Theorem~\ref{thm-extension} on germs defined on a closed subset $S\subset X$? For this construction we need to define {\it smooth K-maps at $S$}. More precisely, generalizing the definition of germs at a point, we will say that maps $f_1$ and $f_2$ from neighborhoods $U_1$ and $U_2$ of $S$ into $Y$ are {\it equivalent}, if they coincide in a (smaller) neighborhood of $S$.
Every equivalence class is called a {\it germ at $S$}. We pose the same question. Given a $C^q$-germ at $S$ does there exist a global representative? If we assume that the $C^q$-bump functions at $S$ exist and equal to identity in a neighborhood  of $S$, then the answer is positive. Moreover, let there exist a $C^q$-map $H:X \to X$ whose image $H(X)$ is contained in a neighborhood  $U$ of $S$
and which is equal to the identity map in a smaller neighborhood. Then every local map $f$ defined in $U$ can be extended on the whole $X$.  It suffices to set $F(x)=f(H(x))$.
\begin{example}Assume that $X$ has a $C^q$-K-map $H$ at $0$, and let $S=B_r (z)$ be a ball of radius $r$ centered at $z$. Let
               $$         U=\{x:||x||<r+\delta\}
$$
be an open ball. Then, under an appropriate choice of a constant $c$,  the map
                   $$H_1(x)=z+\frac{1}{c}H(c(x-z))$$
is the identity map in a neighborhood of $B_r$, while its image is contained in $U$. Therefore, every $C^q$-germ at $B_r (z)$ contains a global representative.
\end{example}
So, the question is for which pairs $(S,X)$ do similar constructions exist? In particular, can a smooth K-map be constructed for bounded $S$? For example, we do not know whether a smooth K-map can be constructed for a sphere $S=\{x\in C[0,1]: ||x||=r\}$. 

The same question can be asked for $S$ being an arbitrary subspace of $X=C[0,1]$. If the answer is positive, then the Banach-Mazur theorem will imply that any separable Banach space possesses a smooth K-map.

\item[3.] The Borel lemma for finite-dimensional spaces is a particular case of the well-known Whitney extension theorem from a closed set $S\subset \R^n$. What is an infinite-dimensional version of the Whitney theorem? 
\end{itemize}

\bibliographystyle{amsalpha}

\end{document}